\documentclass[11pt]{article}

\usepackage{amsxtra,amssymb,amsthm,amsmath,latexsym}
\usepackage{graphicx}
\usepackage{epsfig}
\usepackage{afterpage}
\newtheorem{thm}{Theorem}[]

\newtheorem{rem}[]{Remark}
 \newcommand{\thmref}[1]{Theorem~\ref{#1}}

\newcommand{\R}{{\mathbb R}}

\newcommand{\dl}{{\delta}}
\newcommand{\bee}{\begin{equation*}}
\newcommand{\eee}{\end{equation*}}
\newcommand{\be}{\begin{equation}}
\newcommand{\ee}{\end{equation}}
\newcommand{\pn}{\par\noindent}

\title{A nonlinear inequality and evolution problems}
\author{A.G. Ramm \\
\small Department of Mathematics\\[-0.8ex]
\small Kansas State University, Manhattan, KS 66506-2602, USA\\
\small \texttt{ramm@math.ksu.edu}}

\date{}
\begin{document}

Journal of Inequalities and Special Functions (JIASF), 1, N1, (2010), 1-9.

\maketitle

\begin{abstract}
Assume that $g(t)\geq 0$,  and  $$\dot{g}(t)\leq
-\gamma(t)g(t)+\alpha(t,g(t))+\beta(t),\ t\geq 0;\quad g(0)=g_0;\quad
\dot{g}:=\frac{dg}{dt}, $$
on any interval $[0,T)$ on which $g$ exists and has bounded derivative
from the right, $\dot{g}(t):=\lim_{s\to +0}\frac{g(t+s)-g(t)}{s}$.  
It is assumed that  $\gamma(t)$, 
and $\beta(t)$ are nonnegative continuous 
functions of $t$ defined on
$\R_+:=[0,\infty)$,  the function $\alpha(t,g)$ is
defined for all $t\in \R_+$,  locally Lipschitz 
with
respect to $g$ uniformly with respect to $t$ on any compact subsets
$[0,T]$, $T<\infty$,   and non-decreasing with respect to $g$, 
$\alpha(t,g_1)\geq \alpha(t,g_2)$ if $g_1\ge g_2$. If 
there
exists a function $\mu(t)>0$,  $\mu(t)\in C^1(\R_+)$, such that
$$\alpha\left(t,\frac{1}{\mu(t)}\right)+\beta(t)\leq
\frac{1}{\mu(t)}\left(\gamma(t)-\frac{\dot{\mu}(t)}{\mu(t)}\right),\quad
\forall t\ge 0;\quad \mu(0)g(0)\leq 1,$$ then $g(t)$ exists on all of 
$\R_+$, that is $T=\infty$,  and the following estimate holds:
$$0\leq g(t)\le \frac 1{\mu(t)},\quad \forall t\geq 0. $$ 
If $\mu(0)g(0)< 1$, then $0\leq g(t)< \frac 1{\mu(t)},\quad \forall t\geq 
0. $

A discrete version of this  result is obtained.

The nonlinear inequality, obtained in this paper,  is used
in a study of the Lyapunov stability and asymptotic stability
of solutions to differential
equations in finite and infinite-dimensional spaces.
\end{abstract}
\pn{

 }\\

\section{Introduction}
The goal of this paper is to give a self-contained proof of an
estimate for solutions of a nonlinear inequality \be\label{e1}
\dot{g}(t)\leq -\gamma(t)g(t)+\alpha(t,g(t))+\beta(t),\ t\geq 0;\
g(0)=g_0;\ \dot{g}:=\frac{dg}{dt},  \ee
and to demonstrate some of its many possible applications.

Denote $\R_+:=[0,\infty)$.
It is  not assumed a priori that solutions $g(t)$ to inequality
\eqref{e1} are defined on all of $\R_+$, that is, that  these solutions 
exist globally. We give sufficient conditions for the global existence of 
$g(t)$.
Moreover, under these conditions  a bound on $g(t)$ is given, see
estimate \eqref{e5} in Theorem 1. 
This bound yields the relation $\lim_{t\to \infty}g(t)=0$
if $\lim_{t\to \infty}\mu(t)=\infty$ in \eqref{e5}. 

Let us formulate
our assumptions.

\noindent {\it Assumption A).} We assume that the function $g(t)\geq 0$
is defined on some interval $[0,T)$, has a bounded derivative
$\dot{g}(t):=\lim_{s\to +0}\frac{g(t+s)-g(t)}{s}$ from the 
right at any point of this interval, and  $g(t)$ satisfies 
inequality \eqref{e1} at all $t$ at which $g(t)$ is defined. The 
functions 
$\gamma(t)$, and $\beta(t)$,  are continuous, non-negative,
defined on all of $\R_+$.  The function 
$\alpha(t,g)\ge 0$ is 
continuous on $\R_+\times \R_+$, nondecreasing with respect to $g$,
and locally Lipschitz with respect to $g$. This 
means that $\alpha(t,g)\ge \alpha(t,h)$ if $g\ge h$, and 
\be\label{e2}
|\alpha(t,g)-\alpha(t,h)|\leq L(T,M)|g-h|,
\ee
if $t\in[0,T]$, $|g|\leq M$ and $|h|\leq  M$, $M=const>0$,
where $L(T,M)>0$ is a constant independent of $g$, $h$, and $t$.

\noindent {\it Assumption B).} There exists a $C^1(\R_+)$
function $\mu(t)>0$,  such that \be\label{e3}
\alpha\left(t,\frac{1}{\mu(t)}\right)+\beta(t)\leq
\frac{1}{\mu(t)}\left(\gamma(t)-\frac{\dot{\mu}(t)}{\mu(t)}\right),\quad 
\forall t\ge 0, \ee
\be\label{e4} \mu(0)g(0)< 1. \ee

If $\mu(0)g(0)\le 1$, then the 
inequality sign $< \frac 1 {\mu(t)}$ in Theorem 
1, in formula \eqref{e5}, is replaced by $\le \frac 1 {\mu(t)}$. 

Our results are formulated in Theorems 1 and 2, and {\it Propositions 
1,2}.
{\it Proposition 1} is related to Example 1, and {\it Proposition 2}
is related to Example 2, see below.

\begin{thm}\label{thm1}
If Assumptions A) and B) hold, then any solution $g(t)\ge 0$ 
to inequality \eqref{e1} exists on all of $\R_+$, i.e., $T=\infty$, 
and satisfies the following estimate:
\be\label{e5}0\leq g(t)<\frac{1}{\mu(t)}\quad \forall t\in \R_+. \ee
If $ \mu(0)g(0)\le 1$, then $0\leq g(t)\le \frac{1}{\mu(t)}\quad \forall 
t\in \R_+.$
\end{thm}
\begin{rem}\label{rem1}
If $\lim_{t\to \infty} \mu(t)=\infty$, then $\lim_{t\to
\infty}g(t)=0$. 
\end{rem}
Let us explain how one applies estimate \eqref{e5} in various
problems (see also papers \cite{R593}, \cite{R558},  and the monograph 
\cite{R499}
for other applications of differential inequalities which are
particular cases of inequality \eqref{e1}).

\noindent {\it Example 1.} Consider the problem 
\be\label{e6}
\dot{u}=A(t)u+B(t)u,\quad u(0):=u_0, \ee 
where $A(t)$ is a linear
bounded operator in a Hilbert space $H$ and $B(t)$ is a bounded
linear operator such that 
$$\int_0^\infty \|B(t)\|dt:=C<\infty.$$
Assume that 
\be\label{e7} \text{Re}(A(t)u,u)\leq 0\quad \forall u\in
H,\ \forall t\geq 0.\ee
Operators satisfying inequality \eqref{e7} are called {\it dissipative}.
They arise in many applications, for example in a study of passive
linear and nonlinear networks (e.g., see \cite{R129}, and \cite{R118},
Chapter 3).

One may consider some classes of unbounded linear operator using 
the scheme developed in the proofs of {\it Propositions 1,2}.  For 
example,
in {\it Proposition 1} the operator $A(t)$ can be a generator of $C_0$ 
semigroup
$T(t)$ such that $\sup_{t\ge 0}\|T(t)\|\le m$, where $m>0$ is a constant.

Let $A(t)$ be a linear closed, densely defined in $H$, dissipative 
operator, with domain of definition $D(A(t))$ independent of $t$, and $I$ 
be the 
identity operator in $H$. Assume that
the Cauchy problem    
$$\dot{U}(t)=A(t)U(t),\quad U(0)=I,$$
for the operator-valued function $U(t)$
has a unique global solution
and  
$$\sup_{t\ge 0}\|U(t)\|\le m,$$ where $m>0$ is a constant.
Then such an unbounded operator $A(t)$ can be used in {\it Example 1}. 

\noindent {\it Proposition 1.} {\it If condition  \eqref{e7} holds and  
$C:=\int_0^\infty \|B(t)\|dt<\infty$, then the 
solution to problem \eqref{e6} exists on $\R_+$, is unique,  and 
satisfies the following inequality:
\be\label{e8}
\sup_{t\geq 0}\|u(t)\|\leq e^C\|u_0\|. \ee
}

Inequality \eqref{e8} implies Lyapunov stability of the zero solution
to equation \eqref{e6}. 

Recall that the zero solution to
equation \eqref{e6} is called Lyapunov stable if for any $\epsilon>0$,
however small, one can find a $\delta=\delta(\epsilon)>0$, such that if
$\|u_0\|\le \delta$, then the solution to Cauchy problem
\eqref{e6} satisfies the estimate $\sup_{t\ge 0}\|u(t)\|\le
\epsilon$. If, in addition, $\lim_{t\to \infty}\|u(t)\|=0$, then
the zero solution to equation \eqref{e6} is called asymptotically stable
in the Lyapunov sense.

\noindent {\it Example 2.} Consider an abstract nonlinear
evolution problem 
\be\label{e9}
\dot{u}=A(t)u+F(t,u)+b(t),\quad  u(0)=u_0,\ee 
where $u(t)$ is a function with values in a Hilbert space $H$, $A(t)$ 
is a linear bounded operator in $H$ which satisfies inequality 
\be\label{e10} \text{Re}(Au,u)\leq
-\gamma(t)\|u\|^2,\quad t\geq 0;\qquad \gamma=\frac{r}{1+t}, \ee  
$r>0$ is a constant,  
$F(t,u)$ is a nonlinear map in $H$, and the following estimates hold:
\be\label{e11} \|F(t,u)\|\leq \alpha(t,g),\quad g:=g(t):=\|u(t)\|; \quad
\|b(t)\|\leq \beta(t), 
\ee
where $\beta(t)\ge 0$ and $\alpha(t,g)\ge 0$ 
satisfy the conditions in {\it Assumption A)}. 

Let us assume that \be\label{e12}
\alpha(t,g)\leq c_0g^p,\quad p>1;\quad \beta(t)\leq
\frac{c_1}{(1+t)^{\omega}}, \ee 
where $c_0$, $p$, $\omega$ and $c_1$
are positive constants.\\

\noindent {\it Proposition 2.} {\it If conditions \eqref{e9}-\eqref{e12} 
hold, 
and inequalities  \eqref{e20},\eqref{e21} and \eqref{e23} are satisfied
(see these inequalities in the proof of {\it Proposition 2}),
then
the solution to the evolution problem  \eqref{e9} exists on all of $\R_+$ 
and satisfies the following estimate:
\be\label{e13} 0\leq \|u(t)\|\leq \frac{1}{\lambda(1+t)^q},\qquad
\forall t\geq 0, \ee where $\lambda$ and $q$ are some positive 
constants the choice of which is specified by inequalities  
\eqref{e20},\eqref{e21} and \eqref{e23}.}

The choice of $\lambda $ and $q$ is motivated and explained in the proof 
of {\it Proposition 2}
(see inequalities \eqref{e20}, \eqref{e21} and \eqref{e23} in
Section 2).

Inequality \eqref{e13} implies asymptotic stability of the solution
to problem \eqref{e9} in the sense of Lyapunov and, additionally, gives 
a rate of convergence of $\|u(t)\|$ to zero as $t\to \infty$.

The results in {\it Examples 1,2} can be obtained in Banach space,
but we do not go into detail. 

Proofs of Theorem 1 and  {\it Propositions 1} and {2 } are given in 
Section 2.
Theorem 2, which is a discrete analog of Theorem 1, is formulated and 
proved in Section 3.

\section{Proofs}
\noindent {\it Proof of Proposition 1.} Local existence of the solution 
$u(t)$ 
to problem  \eqref{e6} is known (see, e.g., \cite{DK}). Uniqueness of 
this solution follows from the linearity of the problem and from 
estimate  \eqref{e8}. Let us prove this estimate.  

 Multiply \eqref{e6} by
$u(t)$, let $g(t):=\|u(t)\|$, take real part, use \eqref{e7}, and get
$$\frac{1}{2}\frac{d g^2(t)}{dt}\leq \text{Re}(B(t)u(t),u(t))\leq
\|B(t)\|g^2(t).$$ This implies $g^2(t)\leq g^2(0)e^{2C}$, 
so \eqref{e8} follows. {\it Proposition 1} is proved. \hfill $\Box$

\noindent {\it Proof of Proposition 2.} The local existence and uniqueness 
of 
the solution $u(t)$ to  problem \eqref{e9} follow from {\it Assumption 
A} (see, e.g., \cite{DK}). The existence of  $u(t)$ for all $t\ge 0$,
that is, the global existence of  $u(t)$, follows from estimate 
\eqref{e13} (see, e.g., \cite{R499}, pp.167-168).

Let us derive estimate \eqref{e13}.
Multiply \eqref{e9} by
$u(t)$, let $g(t):=\|u(t)\|$, take real part, use
\eqref{e10}-\eqref{e12} and get \be\label{e14} g\dot{g}\leq
-\gamma(t)g^2(t)+\alpha(t,g(t))g(t)+\beta(t)g(t),\ t\geq 0. \ee 
Since $g\ge 0$, one obtains from this inequality inequality 
\eqref{e1}. However, first we would like to explain in detail
the meaning of the derivative $\dot{g}$ in our proof.

By $\dot{g}$ the right derivatives is understood: 
$$\dot{g}(t):=\lim_{s\to +0}\frac{g(t+s)-g(t)}{s}.$$ 
If $g(t)=\|u(t)\|$
and $u(t)$ is continuously differentiable, then
$\psi(t):=g^2(t)=(u(t),u(t))$ is continuously differentiable, and its 
derivative at the point $t$ at which $g(t)>0$ can be computed by the 
formula: 
$$\dot{g}= Re (\dot{u}(t),u^0(t)),$$ where $u^0(t):=\frac 
{u(t)}{\|u(t)\|}$ . Thus,  
the function $g(t)=\sqrt{\psi(t)}$ is continuously
differentiable at any point at which $g(t)\neq 0$. At a point $t$ at
which $g(t)=0$, the vector $u^0(t)$ is not defined, the
derivative of $g(t)$ does not exist in the usual sense, but the right 
derivative of $g(t)$ still exists and can be calculated explicitly:
\bee\begin{split} \dot{g}(t)&=\lim_{s\to
+0}\frac{\|u(t+s)\|-\|u(t)\|}{s}=\lim_{s\to
+0}\frac{\|u(t)+s\dot{u}(t)+o(s)\| }{s}\\
&=\lim_{s\to 0}\|\dot{u}(t)+o(1)\|=\|\dot{u}(t)\|.
\end{split}\eee
If $u(t)$ is continuously differentiable at some point $t$,
 and $u(t)\neq 0$, then
$$\dot{g}=\|u(t)\|^.\leq \|\dot{u}(t)\|.$$ 
Indeed, 
\bee
2g(t)\dot{g}(t)=(\dot{u}(t),u(t))+(u(t),\dot{u}(t))\leq
2\|\dot{u}\|\|u\|=2\|\dot{u}(t)\|g(t). \eee 
If $g(t)\neq 0$, then
the above inequality implies $\dot{g}(t)\leq \|\dot{u}(t)\|$, as
claimed. One can also derive this inequality from the formula
$\dot{g}= Re (\dot{u}(t),u^0(t))$, since $|Re (\dot{u}(t),u^0(t))|\leq 
\|\dot{u}(t)\|$. 

If $g(t)>0$, then from \eqref{e14} one obtains
\be\label{e15} \dot{g}(t)\leq
-\gamma(t)g(t)+\alpha(t,g(t))+\beta(t),\quad t\ge 0. \ee If $g(t)=0$ on
an open set, then inequality \eqref{e15} holds on this set also,
because $\dot{g}=0$ on this set while the right-hand side of
\eqref{e15} is non-negative at $g=0$. If $g(t)=0$ at some point
$t=t_0$, then \eqref{e15} holds at $t=t_0$ because, as we have proved
above, $\dot{g}(t_0)=0$, while the right-hand side of \eqref{e15} is
equal to $\beta(t)\ge 0$ if $g(t_0)=0$, and is, therefore, non-negative if 
$g(t_0)=0$. 

If assumptions \eqref{e12} hold, then inequality \eqref{e15} can be 
rewritten as
\be\label{e16} \dot{g}\leq
-\frac{1}{(1+t)^\nu}g+c_0g^p+\frac{c_1}{(1+t)^\omega},\quad p>1. \ee Let
us look for $\mu(t)$ of the form \be\label{e17}
\mu(t)=\lambda(1+t)^q,\quad q=const>0,\quad \lambda=const>0. \ee
Inequality \eqref{e3} takes the form 
\be\label{e18}
\frac{c_0}{[\lambda(1+t)^q]^p} +\frac{c_1}{(1+t)^\omega}\leq
\frac{1}{\lambda(1+t)^q}\left(
\frac{r}{(1+t)^\nu}-\frac{q}{1+t}\right),\quad t>0,\ee 
or 
\be\label{e19}
\frac{c_0}{\lambda^{p-1}(1+t)^{q(p-1)}
}+\frac{c_1\lambda}{(1+t)^{\omega-q}}+\frac{q}{1+t}\leq
\frac{r}{(1+t)^\nu},\quad t>0 \ee 
Assume that the following inequalities \eqref{e20}-\eqref{e21} hold:
\be\label{e20} q(p-1)\geq
\nu,\quad \omega -q\geq \nu,\quad 1\geq \nu, \ee 
and 
\be\label{e21}
\frac{c_0}{\lambda^{p-1}}+c_1\lambda+q\leq r. \ee

Then inequality \eqref{e19} holds, and \thmref{thm1} yields \be\label{e22}
g(t)=\|u(t)\|<\frac{1}{\lambda(1+t)^q},\quad \forall t\geq 0, \ee 
provided
that \be\label{e23} \|u_0\|<\frac{1}{\lambda}. \ee 
Note that for any $\|u_0\|$
inequality \eqref{e23} holds if $\lambda$ is sufficiently large.
For a fixed $\lambda$, however large, inequality \eqref{e21} holds if $r$ 
is sufficiently large. 

{\it Proposition 2} is proved.\hfill $\Box$

The proof of {\it Proposition 2} provides a flexible
general scheme for obtaining estimates of the behavior of the solution
to evolution problem \eqref{e9} for $t\to \infty$.\\

\noindent {\it Proof of \thmref{thm1}.} Let 
\be\label{e24}
g(t)=\frac{v(t)}{a(t)},\quad a(t):=e^{\int_0^t\gamma(s)ds}, \ee
\be\label{e25} \eta(t):=\frac{a(t)}{\mu(t)},\quad
\eta(0)=\frac{1}{\mu(0)}>g(0). \ee Then
inequality \eqref{e1} reduces to 
\be\label{e26} \dot{v}(t)\leq
a(t)\alpha\left(t,\frac{v(t)}{a(t)}\right)+a(t)\beta(t),\quad t\geq 0; 
\quad
v(0)=g(0). \ee One has 
\be\label{e27}
\dot{\eta}(t)=\frac{\gamma(t)a(t)}{\mu(t)}-\frac{\dot{\mu}(t)a(t)}{\mu^2(t)}
=\frac{a(t)}{\mu(t)}\left(\gamma(t)-\frac{\dot{\mu}(t)}{\mu(t)}\right).
\ee From \eqref{e3}, \eqref{e24}-\eqref{e27}, one gets
\be\label{e28} v(0)<\eta(0),\quad \dot{v}(0)\leq \dot{\eta}(0). \ee
Therefore there exists a $T>0$ such that 
\be\label{e29}
0\leq v(t)<\eta(t),\quad \forall t\in[0,T). \ee Let us prove that
$T=\infty$. 

First, note that if inequality \eqref{e29} holds for
$t\in[0,T)$, 
or, equivalently, if
\be\label{e30}
0\leq g(t)<\frac 1 {\mu(t)}, \qquad \forall t\in[0,T),\ee 
then 
\be\label{e31}\dot{v}(t)\leq \dot{\eta}(t),\qquad \forall t\in[0,T).\ee
One can pass to the limit $t\to T-0$ in this inequality and get
\be\label{e32} \dot{v}(T)\leq \dot{\eta}(T). \ee
Indeed, from inequality  \eqref{e30} it follows that
$$\alpha\left(t,\frac {v}{a}\right)+\beta=\alpha(t,g)+\beta\leq 
\alpha(t,\frac 1 {\mu})+\beta,$$
because $\alpha(t,g)\le \alpha(t,\frac 1 {\mu})$.

Furthermore, from inequality \eqref{e3} one derives:
$$\alpha\left(t,\frac 1 {\mu}\right)+\beta\leq\frac 1 
{\mu(t)}\left(\gamma(t)-\frac{\dot{\mu}(t)}{\mu(t)}\right).$$

Consequently, from inequalities \eqref{e26}-\eqref{e27} one obtains
$$ \dot{v}(t)\leq
\frac{a(t)}{\mu(t)}\left(\gamma(t)-\frac{\dot{\mu}(t)}{\mu(t)}\right)
=\dot{\eta}(t),\qquad
t\in[0,T), $$
and inequality \eqref{e31} is proved.

Let $t\to T-0$ in \eqref{e31}. The function $\eta(t)$
is defined for all $t\in \R_+$ and $\dot{\eta}(t)$ is continuous on
$\R_+$. Thus, there exists the limit
$$\lim_{t\to T-0}\dot{\eta}(t)=\dot{\eta}(T).$$
By $\dot{v}(T)$ in inequality \eqref{e32} one may understand 
$\limsup_{t\to T-0}\dot{v}(t)$,
which does exist because $\dot{v}(t)$ is bounded for all $t<T$
by a constant independent of $t\in [0,T]$, due to the estimate 
\eqref{e31}. 

To prove that $T=\infty$ we prove that the "upper" solution
$w(t)$ to the inequality \eqref{e26} exists for all $t\in \R_+$.

Define $w(t)$ as the solution to
the problem 
\be\label{e33}
\dot{w}(t)=a(t)\alpha\left(t,\frac{w(t)}{a(t)}\right)+a(t)\beta(t),\quad
w(0)=v_0. \ee 
The unique solution to problem \eqref{e33} exists locally, on $[0,T)$, 
because $\alpha(t,g)$ is assumed locally Lipschitz. On the interval
$[0,T)$ one obtains inequality 
$$ 0\le v(t)\leq w(t), \qquad t\in[0,T),$$ 
by the standard comparison lemma
(see, e.g., \cite{R499}, p.99, or \cite{H}). Thus, inequality
\be\label{e34} 0\le v(t)\leq w(t)\leq \eta(t),\qquad t\in[0,T), \ee holds.

The desired conclusion $T=\infty$ one derives from
the following claim:

\noindent {\it Proposition 3.} {\it The solution $w(t)$ 
to problem \eqref{e33} exists on every
interval $[0,T]$ on which it is a priori bounded by 
a constant depending only on $T$.} 

We prove this claim later. 
Assuming that this claim is established, one concludes that
$T=\infty$.
Let us
finish the proof of \thmref{thm1} using {\it Proposition 3}. 

Since $\eta(t)$
is bounded on any interval $[0,T]$ ( by a constant depending only on  $T$)
one concludes from {\it Proposition 3} that $w(t)$ ( and, therefore, 
$v(t)$)
exists on all of $\R_+$. If $v(t)\leq \eta(t)$ $\forall t\in \R_+$, then
inequality \eqref{e5} holds (see \eqref{e24} and \eqref{e25}), and
\thmref{thm1} is proved.

 Let us prove {\it Proposition 3}.

\noindent {\it Proof of Proposition 3.} We prove a more general
statement, namely, {\it Proposition 4},  from which {\it Proposition 3} 
follows.\\

\noindent {\it Proposition 4.} {\it Assume that \be\label{e35}
\dot{u}=f(t,u),\quad u(0)=u_0, \ee where $f(t,u)$ is an operator
in a Banach space $X$, locally Lipschitz with respect to $u$ for
every $t$, i.e., $\|f(t,u)-f(t,v)\|\leq L(t,M)\|u-v\|$, $\forall
v,v\in \{u\ :\ \|u\|\leq M\}$. The unique solution to problem \eqref{e35}
exists for all $t\ge 0$ if and only if 
\be\label{e36} \|u(t)\|\leq
c(t), \quad t\geq 0, \ee where $c(t)$ is a continuous function
defined for all $t\geq 0$, and inequality \eqref{e36}
holds for all $t$ for which $u(t)$ exists.}

\noindent {\it Proof of Proposition 4.} The necessity of condition
\eqref{e36} is obvious: one may take $c(t)=\|u(t)\|$. 

To prove its
sufficiency, recall a known local existence theorem, 
see, e.g.,  \cite{DK}.

{\it Proposition 5.} {\it If $\|f(t,u)\|\leq M_1$ and 
$\|f(t,u)-f(t,v)\|\leq 
L\|u-v\|$,
$\forall t\in [t_0,t_0+T_1],$ $\|u-u_0\|\leq R,$ $u_0=u(t_0)$, then
there exists a $\dl>0$, $\dl=\min(\frac{R}{M_1}, \frac{1}{L},
T_1-T)$, such that for every $\tau_0\in [t_0,T]$, $T<T_1$, there
exists a unique solution to equation \eqref{e35} in the interval
$(\tau_0-\dl,\tau+\dl)$ and $\|u(t)-u(t_0)\|\leq R.$ 
}

Using   {\it Proposition 5}, let us prove the sufficiency of
the assumption \eqref{e36} for the global existence of $u(t)$,
i.e., for the existence of $u(t)$ for all
$t\ge t_0$. 

Assume that
condition \eqref{e36} holds and the solution to problem \eqref{e35} exists
on $[t_0,T)$ but does not exist on $[t_0,T_1)$ for any $T_1>T$. Let us
derive a contradiction from this assumption. 

 {\it Proposition 5} 
guarantees the existence and uniqueness of the solution to
problem \eqref{e35} with $t_0=T$ and the initial value $u_0=u(T-0).$
The value $u(T-0)$ exists if inequality \eqref{e36} holds, as we
prove below. The solution $u(t)$ exists on the interval
$[T-\dl,T+\dl]$ and, by the uniqueness theorem, coincides with the
solution $u(t)$ of the problem \eqref{e35} on the interval
$(T-\dl,T)$. Therefore, the solution to \eqref{e35} can be
uniquely extended to the interval
$[0,T+\dl)$, contrary to the assumption that it does not exist on the 
interval
$[0,T_1)$ with any $T_1>T$. This contradiction proves that
$T=\infty$, i.e., the solution to problem \eqref{e35} exists for all 
$t\geq t_0$ if estimate \eqref{e36} holds and $c(t)$ is defined and
continuous $\forall t\ge t_0$. 

Let us now prove the existence of the limit 
$$\lim_{t\to T-0}u(t):=u(T-0).$$ 
Let $t_n\to T$, $t_n<T$.
Then \bee \|u(t_n)-u(t_{n+m})\|\leq
\int_{t_n}^{t_{n+m}}\|f(t,u(s))\|ds\leq (t_{n+m}-t_n)M_1\to 0\text{
as }n\to\infty. \eee Therefore, by the Cauchy criterion, there exists
the limit 
$$\lim_{t_n\to T-0}u(t)=u(T-0).$$ Estimate \eqref{e36}
guarantees the existence of the constant $M_1$.

{\it Proposition 4} is proved \hfill   $\Box$

Therefore {\it Proposition 3} is also proved and, consequently, the 
statement of Theorem 1,
corresponding to the assumption \eqref{e5}, is proved. In our case 
$t_0=0$, but one may replace the initial moment $t_0=0$ in
\eqref{e1} by an arbitrary $t_0\in \R_+$.

Finally, if $g(0)\leq \frac 1 {\mu(0)}$,
then one proves the inequality 
$$0\le g(t)\le \frac 1 {\mu(t)}, \qquad
\forall t\in \R_+$$ using the argument similar to the above. 
This argument is left to the reader. 

\thmref{thm1} is proved.\hfill $\Box$

\section{Discrete version of \thmref{thm1}}
\begin{thm}\label{thm2}
Assume that $g_n\geq 0$, $\alpha(n,g_n)\geq 0,$ \be\label{e37}
g_{n+1}\leq (1-h_n\gamma_n)g_n+h_n\alpha(n,g_n)+h_n\beta_n,\quad
h_n>0,\ 0<h_n\gamma_n<1,\ee and $\alpha(n,g_n)\geq \alpha(n,q_n)$ if
$g_n\geq q_n$. If there exists a sequence $\mu_n> 0$ such that
\be\label{e38} \alpha(n,\frac{1}{\mu_n}) +\beta_n\leq
\frac{1}{\mu_n}(\gamma_n-\frac{\mu_{n+1}-\mu_n}{h_n\mu_n}),
\ee
and
\be\label{e39} g_0\leq \frac{1}{\mu_0}, \ee 
then 
\be\label{e40}
0\leq g_n\leq \frac{1}{\mu_n}\qquad \forall n\geq 0. \ee
\end{thm}
\begin{proof}
For $n=0$ inequality \eqref{e40} holds because of \eqref{e39}.
Assume that it holds for all $n\leq m$ and let us check that then it holds
for $n=m+1$. If this is done, \thmref{thm2} is proved. Using the
inductive assumption, one gets: \bee g_{m+1}\leq
(1-h_m\gamma_m)\frac{1}{\mu_m}+h_m\alpha(m,\frac{1}{\mu_m})+h_m\beta_m.
\eee 
This and inequality \eqref{e38} imply: \bee\begin{split} g_{m+1}&\leq
(1-h_m\gamma_m)\frac{1}{\mu_m}+h_m\frac{1}{\mu_m}(\gamma_m-\frac{\mu_{m+1}-
\mu_m}{h_m\mu_m})\\
&=\frac{\mu_mh_m-\mu_mh^2_m\gamma_m+h^2_m\gamma_m\mu_m-h_m\mu_{m+1}+h_m\mu_m
}{\mu^2_mh_m}\\
&=\frac{2\mu_mh_m-h_m\mu_{m+1}}{\mu_m^2h_m}=\frac{2\mu_m-\mu_{m+1}}{\mu^2_m}=
\frac{1}{\mu_{m+1}}+\frac{2\mu_m-\mu_{m+1}}{\mu^2_m}-\frac{1}{\mu_{m+1}}.
\end{split}\eee 
The proof is completed if one checks that
\bee \frac{2\mu_m-\mu_{m+1}}{\mu^2_m}\leq \frac{1}{\mu_{m+1}}, \eee
or, equivalently, that 
\bee 2\mu_m\mu_{m+1}-\mu^2_{m+1}-\mu^2_m\leq 0. 
\eee 
The last
inequality is obvious since it can be written as 
$$-(\mu_m-\mu_{m+1})^2\le 0.$$
\thmref{thm2} is proved.
\end{proof}
\thmref{thm2} was formulated in \cite{R593} and proved in \cite{R558}.
We included for completeness a proof, which is different from the one in 
\cite{R558} only slightly.

\end{document}